\documentclass[10pt,reqno]{amsart}
\usepackage{amsmath,amsfonts,amsthm,amssymb}
\usepackage[shortlabels]{enumitem}
\usepackage{tikz}
\usepackage{soul}
\usepackage{hyperref}

\newcommand{\R}{\mathbb{R}}
\newcommand{\N}{\mathbb{N}}
\newcommand{\id}{\operatorname{id}}

\newcommand{\Fix}{\operatorname{Fix}}
\newcommand{\ext}{\operatorname{ext}}

\newcommand{\spn}{\operatorname{span}}
\newcommand{\supp}{\operatorname{supp}}

\newcommand{\bs}{\setminus}

\theoremstyle{plain}
\newtheorem{theorem}{Theorem}[section]
\newtheorem{lemma}[theorem]{Lemma}

\newtheorem{proposition}[theorem]{Proposition}

\theoremstyle{remark}
\newtheorem{remark}[theorem]{Remark}

\theoremstyle{definition}

\newtheorem{question}[theorem]{Question}

\numberwithin{equation}{section}

\begin{document}
\title[Real analytic nonexpansive maps on polyhedral normed spaces]{Real analytic nonexpansive maps on polyhedral normed spaces}
\author[B. Lins]{Brian Lins}
\date{}
\address{Brian Lins, Hampden-Sydney College}
\email{blins@hsc.edu}
\subjclass[2010]{Primary 47H09, 47H10; Secondary 39A23, 46T20}
\keywords{Nonexpansive maps, real analytic maps, fixed points, periodic orbits, polyhedral norms}

\begin{abstract}
If a real analytic nonexpansive map on a polyhedral normed space has a nonempty fixed point set, then we show that there is an isometry from an affine subspace onto the fixed point set. As a corollary, we prove that for any real analytic 1-norm or $\infty$-norm nonexpansive map on $\mathbb{R}^n$, there is a positive integer $q$ such that the period of any periodic orbit divides $q$ and $q$ is the order, or twice the order, of a permutation on $n$ letters. This confirms Nussbaum's $2^n$ Conjecture for $\infty$-norm nonexpansive maps in the special case where the maps are also real analytic. 
\end{abstract}

\maketitle

\section{Introduction}

Bruck \cite{Bruck73} showed that the fixed point set of a nonexpansive map on a Banach space with the fixed point property is a nonexpansive retract.  This implies that the fixed point set must be connected.  In strictly convex spaces, the fixed point set must also be convex, but that is not true in other Banach spaces, and in general the fixed point sets of nonexpansive maps can be quite complicated.  We show in this paper that the fixed point set of a real analytic nonexpansive map on a finite dimensional Banach space with a polyhedral norm must be a Lipschitz manifold that is isometric to a linear subspace.   

Nonexpansive maps on polyhedral normed spaces arise naturally in nonlinear Perron--Frobenius theory.  In some important cases, these maps are also real analytic. Nussbaum \cite[Section 2]{Nussbaum89} introduced and studied a large class $\mathcal{M}$ of real analytic nonexpansive maps with applications in population biology, see also \cite[Section 6.6]{LemmensNussbaum}.  Lim \cite{Lim05} and Qi \cite{Qi05} independently introduced the H-eigenproblem for nonnegative tensors which seeks entrywise positive eigenvectors for nonlinear maps corresponding to tensors. The eigenvectors of these maps correspond to fixed points of maps on $\R^n$ which are real analytic and $\infty$-norm nonexpansive, see e.g., \cite[Section 7.2]{Lins22}.  Real analytic nonexpansive maps also appear in the theory of neural networks when composing norm one linear transformations with real analytic activation functions, such as the sigmoid function \cite{PiCaGa24}. 


Let $X$ be a real Banach space with norm $\|\cdot\|$ and dual space $X^*$. Let $B_X$ denote the closed unit ball in $X$.  A finite dimensional Banach space has a \emph{polyhedral norm} if the unit ball has a finite number of extreme points, or equivalently if the dual unit ball $B_{X^*}$ has a finite number of extreme points. We use $\ext C$ to denote the extreme points of a convex set $C$.

For any map $f:X \rightarrow X$, let $\Fix(f)$ denote the set of fixed points of $f$. 
A map $f: X \rightarrow X$ is \emph{nonexpansive} if $\|f(x)-f(y)\| \le \|x-y\|$ for every $x, y \in X$.  
A function $f:X \rightarrow X$ is \emph{real analytic} if for every $x \in X$, there is a convergent power series for $f$ in a neighborhood of $x$, that is, there is an $r > 0$ and continuous symmetric $n$-linear forms $A_n: X^n \rightarrow X$ such that $\sum_{n = 1}^\infty \|A_n\| r^n < \infty$ and 
$$f(x + h) = f(x) + \sum_{n = 1}^\infty A_n(h^n)$$ 
for all $h$ in a neighborhood of $0$ in $X$. When $X = \R^n$, this definition is equivalent to requiring that each component of $f$ be a real analytic function.

In this paper, our primary interest is real analytic nonexpansive maps on polyhedral normed spaces.  However, for context we begin with a simple result about the fixed point sets of real analytic nonexpansive maps on strictly convex Banach spaces. A Banach space is \emph{strictly convex} if the boundary of the unit ball does not contain any line segments. 

\begin{proposition}
Let $X$ be a strictly convex Banach space.  Suppose $f: X \rightarrow X$ is nonexpansive and real analytic.  If $\Fix(f)$ is nonempty, then it is an affine subspace of $X$.
\end{proposition}

\begin{proof}
If $\Fix(f)$ is empty or contains only a single point, there is nothing to prove.  Suppose that $\Fix(f)$ contains two distinct points, $x$ and $y$. Let $\delta = \|x-y\|$. For any $0 < t < 1$, the closed balls $x + t \delta B_X$ and $y + (1-t) \delta B_X$ intersect at a single point $(1-t)x + t y$. Since both balls are invariant under $f$ by nonexpansiveness, the intersection points are fixed points for all $0 < t < 1$.  Since $f$ is real analytic, so is the map 
$$t \mapsto f((1-t)x + t y) - (1-t) x - t y.$$
As this map is identically zero for all $0 < t < 1$, it follows that it is identically zero for all $t \in \R$.  Therefore every point on the line $\{ (1-t) x + t y : t \in \R \}$ is a fixed point of $f$.  From this we conclude that the affine hull of $\Fix(f)$ is contained in $\Fix(f)$, and therefore $\Fix(f)$ is an affine subspace.  
\end{proof}


The fixed point set of real analytic nonexpansive maps in a polyhedral normed space need not be an affine subspace. However, in section \ref{sec:main} we state and prove the main result of this paper which guarantees that the fixed point set of a real analytic polyhedral norm nonexpansive map must be isometric to an affine subspace. 

In section \ref{sec:periodic} we consider periodic orbits of real analytic nonexpansive maps. 
The \emph{orbit} of a point $x \in X$ under the action of a map $f: X \rightarrow X$ is the set $\{f^k(x) : k \ge 0\}$.  A point $x \in X$ is a \emph{periodic point} of $f$ if its orbit is finite.  In that case, the \emph{minimal period} of $x$, also referred to as the \emph{period} of the orbit of $x$, is the number of elements in its orbit. The following proposition summarizes several known facts about the periodic points of polyhedral norm nonexpansive maps \cite[Theorem 4.2.1, Corollary 4.2.5, and Lemma 4.2.6]{LemmensNussbaum}.  

\begin{proposition} \label{prop:polyhedral}
For any polyhedral normed space $X$, there is an upper bound $M$ on the minimal periods of periodic points under the action of any nonexpansive map on $X$.  Moreover, if $f: X \rightarrow X$ is nonexpansive and $\Fix(f)$ is not empty, then for every $x \in X$, there is a periodic point $\xi$ of $f$ with minimal period $q$ such that $\lim_{k \rightarrow \infty} (f^{kq}(x)) = \xi$.
\end{proposition}

Nussbaum has conjectured \cite{Nussbaum90} that any periodic point of an $\infty$-norm nonexpansive map on $\R^n$ has minimal period at most $2^n$ \cite[Conjecture 4.2.2]{LemmensNussbaum}. The conjecture has been confirmed only up to $n = 3$ \cite{LyNu92}, and the proof for $n = 3$ is nontrivial. The current best known upper bound on the minimal periods of $\infty$-norm nonexpansive maps is $\max_{0 \le k \le n} 2^n \binom{n}{k}$ \cite{LeSc05}.  The conjectured upper bound of $2^n$ would be tight, since there are $\infty$-norm nonexpansive maps on $\R^n$ with periodic orbits containing $2^n$ elements.  In section \ref{sec:periodic}, we show that the periodic orbits of real analytic 1-norm and $\infty$-norm nonexpansive maps on $\R^n$ correspond to the periodic orbits of norm one linear maps. Then, by a result of Lemmens and van Gaans, the maximum period of any orbit is at most twice the order of a permutation on $n$ letters. This implies Nussbaum's $2^n$ Conjecture for real analytic $\infty$-norm nonexpansive maps on $\R^n$.

\section{Main Result} \label{sec:main}

A \emph{retract}, also known as a \emph{projection} is a map $R: X \rightarrow X$ such that $R^2(x) = R(x)$ for all $x \in X$.  We will reserve the term projection for retracts that are also linear.  If $X$ is a finite dimensional normed space, $f: X \rightarrow X$ is nonexpansive, and $\Fix(f)$ is nonempty, then the Krasnoselskii iterates, defined recursively by
$$x^{k+1} = \tfrac{1}{2}f(x^k) + \tfrac{1}{2}x^k,$$
converge to a fixed point for every $x^0 \in X$ \cite{Ishikawa76}.  Define $R: X \rightarrow \Fix(f)$ by 
\begin{equation} \label{eq:retract}
R(x) = \lim_{k \rightarrow \infty} (\tfrac{1}{2}f + \tfrac{1}{2}\id)^k(x)
\end{equation}
where $\id$ denotes the identity map on $X$. Since $R$ is the pointwise limit of the nonexpansive maps $(\tfrac{1}{2}f + \tfrac{1}{2} \id)^k$, it follows that $R$ is a nonexpansive retract onto $\Fix(f)$.  


\begin{theorem} \label{thm:main}
If $X$ is a polyhedral normed space, $f: X \rightarrow X$ is nonexpansive and real analytic, and $\Fix(f)$ is nonempty, then there is an affine subspace $W \subseteq X$ such that the nonexpansive retract $R$ defined by \eqref{eq:retract} is an isometry from $W$ onto $\Fix(f)$.  
\end{theorem}

In order to prove this result, we begin with the following observations. For each of the following lemmas, we assume that $X$ is a polyhedral normed space and $f:X \rightarrow X$ is real analytic and nonexpansive.  
Let $J : X \rightrightarrows B_{X^*}$ be the \emph{duality map}
$$J(x) = \{ \phi \in B_{X^*} : \|x\| = \phi(x) \}.$$
Note that $J$ maps points in $X$ to faces of $B_{X^*}$.

For any set $E \subseteq B_{X^*}$ and any map $f: X \rightarrow X$, let 
$$S_E(f) = \{x \in X : \phi(f(x)) = \phi(x) \text{ for all } \phi \in E \},$$
and 
$$L_E = \{x \in X : \phi(x) = \psi(x) \text{ for all } \phi, \psi \in E \}.$$

\begin{lemma} \label{lem:duality}
For any $x \in X$, there exists $\epsilon > 0$ such that $J(x+y) \subseteq J(x)$ for every $y \in X$ with $\|y \|\le \epsilon$.  
\end{lemma}

\begin{proof}
The claim is trivial if $x = 0$, since $J(x) = B_{X^*}$ in that case.  Suppose that $x \ne 0$. Since $X$ is a polyhedral normed space, the dual unit ball $B_{X^*}$ only has a finite number of extreme points.  Therefore there is a constant $c < 1$ such that $\phi(x) < c \|x\|$ for all $\phi \in \ext B_{X^*} \bs J(x)$.  Since $c < 1$, we can choose $\epsilon$ small enough so that $c \|x \|+\epsilon < \|x\| - \epsilon$. Then
$$\phi(x + y) \le c \|x \| + \epsilon < \|x \| - \epsilon \le \|x + y \|$$
for all $y \in B_X$ and $\phi \in \ext B_{X^*} \bs E$.  Therefore $J(x+y)$ is a face of $B_{X^*}$ which does not contain any extreme points of $B_{X^*}$ outside $J(x)$.  Since every face of $B_{X^*}$ is the convex hull of the elements of $\ext B_{X^*}$ which are contained in the face \cite[Theorem 7.3]{Bronsted}, we conclude that $J(x+y) \subseteq J(x)$. 
\end{proof}

\begin{lemma} \label{lem:faces} 
If $E = J(y)$ for some $y \in X$, then $L_E$ is the linear span of the face 
$$F_E = \{x \in B_X : \phi(x) = 1 \text{ for all } \phi \in E\}.$$
\end{lemma}

\begin{proof}
It is clear from the definition that $\spn F_E \subseteq L_E$. Let $x \in L_E$. By choosing $\epsilon > 0$ small enough, we can guarantee that $J(y - \epsilon x)$ and $J(y + \epsilon x)$ are subsets of $E$ by Lemma \ref{lem:duality}.  Since both $y + \epsilon x$ and $y - \epsilon x$ are in $L_E$, we have $J(y - \epsilon x) = J(y+\epsilon x) = E$.  Therefore 
$$\frac{y - \epsilon x}{\|y - \epsilon x\|} \text{ and } \frac{y + \epsilon x}{\| y + \epsilon x\|}$$
are both contained in $F_E$, and $x$ is a linear combination of these two points.  Therefore $L_E = \spn F_E$.  
\end{proof}

\begin{lemma} \label{lem:claim1}
If there exist $v, w \in S_E(f)$ such that $J(v-w) = E$, then 
$$w + L_E \subseteq S_E(f).$$
\end{lemma}

\begin{proof}
Let $x = \tfrac{1}{2}(v + w)$. Observe that $J(x-w) = J(v-x) = E$.  By Lemma \ref{lem:duality}, there is a sufficiently small relatively open neighborhood $U$ around $x$ in $w + L_E$ such that $J(u-w)$ and $J(v-u)$ are subsets of $E$ for every $u \in U$. Then, since $u \in w+L_E$, it follows that $J(u-w) = J(v-u) = E$.  For any $\phi \in E$ and $u \in U$,
\begin{align*}
\phi(v) - \phi(u) &= \|v-u\| & \\
&\ge \|f(v)-f(u)\| & (\text{nonexpansiveness}) \\
&\ge \phi(f(v)) - \phi(f(u)) & (\text{since } \phi \in B_{X^*}) \\
&= \phi(v) - \phi(f(u)), & (\text{since } v \in S_E(f))
\end{align*}
which implies that $\phi(f(u)) \ge \phi(u)$. Similarly, 
\begin{align*}
\phi(u) - \phi(w) &= \|u-w\| & \\ 
&\ge \|f(u)-f(w)\| & (\text{nonexpansiveness}) \\
&\ge \phi(f(u)) - \phi(f(w)) & (\text{since } \phi \in B_{X^*}) \\
&= \phi(f(u)) - \phi(w). & (\text{since } w \in S_E(f))
\end{align*}
Combining this with the previous inequality, we have $\phi(f(u)) = \phi(u)$ for all $u \in U$ and $\phi \in E$. 
Since $\phi(f(u))-\phi(u)$ is real analytic and identically zero on $U$, it follows that it is identically zero on the affine span of $U$ which is $w+L_E$.  Therefore $w+L_E \subseteq S_E(f)$.
\end{proof} 

We will say that a set $E \subseteq B_{X^*}$ is \emph{locked} if there exist $v, w \in S_E(f)$ such that $J(v-w) = E$.  A locked set $E$ is \emph{minimal} if no proper subset of $E$ is locked.

\begin{lemma} \label{lem:claim2}
If $E \subseteq B_{X^*}$ is a minimal locked set, then $S_E(f) = x + L_E$ for any $x \in S_E(f)$.  
\end{lemma}

\begin{proof} Since $E$ is locked, there exist $v, w \in S_E(f)$ such that $J(v-w) = E$.  Choose any $x \in S_E(f)$. We will show that $x \in w+L_E$. 

Let $y = v-w$. For any sufficiently large $t > 0$, 
$$J(w + ty - x) = J(y + t^{-1}(w-x)) \subseteq E$$
by Lemma \ref{lem:duality}. Let $G = J(w+ty-x)$. Observe that $S_E(f) \subseteq S_G(f)$ since $G \subseteq E$. 
By Lemma \ref{lem:claim1}, $w + L_E \subseteq S_E(f)$. Therefore both $x$ and $w+ty$ are elements of $S_G(f)$. Since $E$ has no locked proper subsets, we conclude that $G = J(w+ty - x) = E$. This means that $x \in w+L_E$ which proves that $S_E(f) = w+L_E$. Of course, $x+L_E = w+L_E$ since $x - w \in L_E$. 
\end{proof}

Let $\mathcal{M}(f)$ be the union of all minimal locked subsets of $B_{X^*}$. The following lemma shows that the linear functionals in $\mathcal{M}(f)$ separate points in $\Fix(f)$.  

\begin{lemma} \label{lem:separate}
If $x, y\in \Fix(f)$ have $\phi(x) = \phi(y)$ for all $\phi \in \mathcal{M}(f)$, then $x = y$. 
\end{lemma}

\begin{proof}
Consider the set $G = J(x-y)$.  If $x \ne y$, then $G$ cannot contain any $\phi \in \mathcal{M}(f)$.  Furthermore, $G$ is locked by definition since both $x, y$ are fixed points of $f$ and are therefore elements of $S_G(f)$. Then $G$ contains a minimal locked subset, but that contracts the assumption that $\mathcal{M}(f)$ is the union of every minimal locked subset. So $x = y$.
\end{proof}

\begin{proof}[Proof of Theorem \ref{thm:main}] 
If $0$ is not a fixed point of $f$, then we can replace $f(x)$ with $f(x+y)-y$ where $y \in \Fix(f)$ which allows us to assume without loss of generality that $0 \in \Fix(f)$. Note that $\Fix(f) \subseteq S_E(f)$ for every $E \subseteq B_{X^*}$. In particular, since $0 \in S_E(f)$ for every minimal locked set $E$ of $f$, Lemma \ref{lem:claim2} implies that $S_E(f) = L_E$. Let $V(f)$ be the intersection of the subspaces $L_E$ corresponding to each minimal locked subset $E$ of $f$.  Then $\Fix(f) \subseteq V(f)$. 

Since the nonexpansive retract $R$ defined by \eqref{eq:retract} is a 1-Lipschitz map from any subspace of $X$ onto $\Fix(f)$, Rademacher's theorem implies that there is a point $u \in V(f)$ where $R$ is differentiable as a map from $V(f)$ into $\Fix(f)$.  Let $A$ be the derivative of $R$ at $u$. 

For any $x \in V(f)$ and $\phi \in \mathcal{M}(f)$ we know that $\phi(f(x)) = \phi(x)$ by the definition of $V(f)$.  We also have $\phi(R(x)) = \phi(x)$ by \eqref{eq:retract}. In addition,
$$\phi( A x) = \lim_{h \rightarrow 0} \frac{\phi(R(u+hx) - R(u))}{h} = \lim_{h \rightarrow 0} \frac{\phi(u+hx) - \phi(u)}{h} = \phi(x).$$
Therefore 
\begin{equation} \label{eq:preserve}
\phi(R(x)) = \phi(Ax) = \phi(x)
\end{equation}
for all $\phi \in \mathcal{M}(f)$ and $x \in V(f)$. 

Let $W$ be the image of $V(f)$ under $A$. As the derivative of a nonexpansive map, $A$ is a nonexpansive linear transformation from $V(f)$ onto $W$. By Lemma \ref{lem:separate}, the functionals in $\mathcal{M}(f)$ separate points in $\Fix(f)$.  We will now prove that these functionals also separate points in $W$.  Suppose that $w \in W$ has $\phi(w) = 0$ for all $\phi \in \mathcal{M}(f)$. By definition, there must be a $v \in V(f)$ such that $w = Av$. By \eqref{eq:preserve} $\phi(v) = \phi(Av) = \phi(w) = 0$ and therefore 
$\phi(R(u+hv) - R(u)) = 0$ for every $\phi \in \mathcal{M}(f)$ and $h \in \R$. Then Lemma \ref{lem:separate} implies that the fixed points $R(u+hv)$ and $R(u)$ must be identical.  So 
$$w = Av = \lim_{h \rightarrow 0} \frac{R(u+hv) - R(u)}{h} = 0.$$
This means that if $x, y \in W$ and $\phi(x) = \phi(y)$ for all $\phi \in \mathcal{M}(f)$, then $x = y$, so $\mathcal{M}(f)$ separates points on $W$. 

Since $A$ and $R$ preserve the values of every $\phi \in \mathcal{M}(f)$ on $V(f)$ and $\mathcal{M}(f)$ separates points in both $W$ and $\Fix(f)$, it follows that $A$ and $R$ are inverse bijections between $W$ and $\Fix(f)$.  Since both maps are nonexpansive, they must also both be isometries. 
\end{proof}

\begin{remark} \label{rem:Asquared}
Let $A$ be the map from the proof above. By \eqref{eq:preserve}, $\phi(A^2 x) = \phi(Ax)$ for all $x \in V(f)$ and $\phi \in \mathcal{M}(f)$. Since $\mathcal{M}(f)$ separates points in $W$, it follows that $A^2 = A$. Therefore $A$ is a nonexpansive linear projection from $V(f)$ onto $W$.  
\end{remark}

\section{Periodic points} \label{sec:periodic}

Let $[n] = \{1, \ldots, n\}$.  Let $e_i$, $i \in [n]$, denote the standard basis vectors for $\R^n$.
Recall that the \emph{$p$-norm} on $\R^n$ is 
$$\|x\|_p = \left( \sum_{i = 1}^n |x_i|^p \right)^{1/p}$$
for $1 \le p < \infty$, and 
$$\|x\|_p = \max_{1 \le i \le n} |x_i|$$
when $p = \infty$.

\begin{theorem} \label{thm:periodic}
Let $p = 1$ or $\infty$. If $f: \R^n \rightarrow \R^n$ is real analytic, $p$-norm nonexpansive, and $\Fix(f)$ is nonempty, then there exists $q \in \N$ such that:
\begin{enumerate}[\normalfont(i)]
\item \label{item:converge} for each $x \in \R^n$, the sequence $(f^{kq}(x))_k$ is convergent;
\item \label{item:periodic} the limit of $(f^{kq}(x))_k$ is a periodic point of $f$ with a minimal period that divides $q$; 
\item \label{item:bound} there exists a permutation $\pi$ on $n$ letters such that $q$ is either the order of $\pi$ or twice the order of $\pi$. 
\end{enumerate}
\end{theorem} 

The conclusions of Theorem \ref{thm:periodic} are known to hold for $1 < p < \infty, p \ne 2$ without requiring $f$ to be real analytic. This is a result of Sine \cite[Theorem 3]{Sine90}, see also \cite{LvG09}. Lemmens and van Gaans have shown that the conclusions of Theorem \ref{thm:periodic} are also true for linear nonexpansive maps when $p = 1$ or $\infty$ \cite[Theorem 2.1]{LvG03}.  We will prove Theorem \ref{thm:periodic} by using Theorem \ref{thm:main} to reduce to the linear case where we can apply the result of Lemmens and van Gaans. 

\begin{proof}[Proof of Theorem \ref{thm:periodic}]
For every $x \in \R^n$, there is a periodic point $\xi \in \R^n$ with minimal period $q_x$ such that $\xi = \lim_{k \rightarrow \infty} f^{k q_x} (x)$ by Proposition \ref{prop:polyhedral}.  Moreover, there is an upper bound on $q_x$ which depends only on $n$ and the norm. Let $q$ be the least common multiple of the minimal periods $q_x$ taken over all $x \in \R^n$.  Then assertions \ref{item:converge} and \ref{item:periodic} follow. Furthermore, every periodic point of $f$ is contained in $\Fix(f^q)$. 

We may assume without loss of generality that $f(0) = 0$. Note that compositions of real analytic maps are real analytic (see e.g., \cite{Whittlesey65}), so in particular $f^q$ is real analytic. Then by the proof of Theorem \ref{thm:main}, there is a nonexpansive retract $R$ from $\R^n$ onto $\Fix(f^q)$ and a nonexpansive linear map $A$ on $V(f^q)$ with range $W$ such that $R$ is an isometry from $W$ onto $\Fix(f^q)$ and $A$ is the inverse isometry from $\Fix(f^q)$ onto $W$. Since $f$ is a surjective isometry on $\Fix(f^q)$, it follows that $A \circ f \circ R$ is a surjective isometry on $W$. By the Mazur--Ulam theorem (see e.g., \cite{Nica12}), $A \circ f \circ R$ is a linear map on $W$.  Since the action of $f$ on $\Fix(f^q)$ is isometric to the action of $A \circ f \circ R$ on $W$, it follows that $q$ is the least common multiple of the minimal periods of the periodic points of $A \circ f \circ R$.   

One delicate issue remains.  In order to reduce to the linear case where the desired bound on the period $q$ is known, we need $A \circ f \circ R$ to extend to a nonexpansive linear map on all of $\R^n$.  We will give separate proofs for $p = 1$ and $p = \infty$ that there are nonexpansive linear projections $P$ from $\R^n$ onto $V(f^q)$.  Since $A$ is a nonexpansive linear projection from $V(f^q)$ onto $W$ by Remark \ref{rem:Asquared}, it will follow that
$$A \circ f \circ R \circ A \circ P$$
is a nonexpansive linear map on $\R^n$ which is identical to $A \circ f \circ R$ on $W$.  Conclusion \ref{item:bound} will then follow from \cite[Theorem 2.1]{LvG03}.

($p=1$ case). Let $X$ be $\R^n$ with the 1-norm. Consider any minimal locked set $E$ for $f^q$.  By Lemma \ref{lem:faces}, $L_E$ is the linear span of the face $F_E$ of the unit ball $B_X$. Since $B_X$ is a polytope, the extreme points of $F_E$ are all extreme points of $B_X$ \cite[Theorem 7.3]{Bronsted}. In particular, $L_E$ is the span of the extreme points of $B_X$ that are contained in $F_E$.  Since the extreme points of $B_X$ are the vectors $\pm e_i$, $i \in [n]$, it follows that $L_E$ is the set of vectors with support in a subset $I_E \subseteq [n]$, that is, 
$$L_E = \{x \in \R^n : x_i = 0 \text{ for all } i \notin I_E \}.$$
Since $V(f^q)$ is the intersection of the subspaces $L_E$ over all minimal locked sets $E$ for $f^q$, it follows that $V(f^q) = \{x \in \R^n : x_i = 0 \text{ for all } i \notin I \}$ where $I$ is the intersection of the sets $I_E$ taken over every minimal locked set $E$.  

Let $P$ be the linear projection
$$(P x)_i = \begin{cases} x_i & \text{ if } i \in I, \\ 0 & \text{ otherwise.} \end{cases}$$
It is clear that $P$ is a 1-norm nonexpansive map from $\R^n$ onto $V(f^q)$, which completes the proof when $p = 1$.

($p = \infty$ case). Let $X$ be $\R^n$ with the $\infty$-norm. We can partition $[n]$ into a union of disjoint sets $I_1, I_2, \ldots, I_m$ where each $I_k$ is a maximal set such that $|x_i| = |x_j|$ for all $x \in V(f^q)$ and $i, j \in I_k$.  Note that at most one $I_k$ can have $x_i = 0$ for all $x \in V(f^q)$ and $i \in I_k$.  For every other $I_k$, there exists $x \in V(f^q)$ such that $|x_i| = 1$ for all $i \in I_k$. 

Consider any minimal locked set $E$ for $f^q$.  Since $E$ is an face of $B_{X^*}$, the extreme points of $E$ are extreme points of $B_{X^*}$ \cite[Theorem 7.3]{Bronsted}.  The extreme points of $B_{X^*}$ are $\pm e_i$, $i \in [n]$. Let 
$$\supp(E) = \{ i \in [n] : e_i \in E \text{ or } -e_i \in E \}$$ 
be the \emph{support} of $E$. Note that $|x_i| = |x_j|$ for every $x \in L_E$ and $i, j \in \supp(E)$.  Since $V(f^q) \subseteq L_E$, it follows that $\supp(E)$ must be a subset of one of the maximal sets $I_k$ in our partition, and it is disjoint from all the other sets in the partition.  


For each $I_k$, choose $x_{I_k} \in V(f)$ such that $(x_{I_k})_i = 1$ for some $i \in I_k$, if any such $x_{I_k}$ exists, otherwise let $x_{I_k} = 0$. Construct $v_{I_k} \in \R^n$ such that 
$$(v_{I_k})_i = \begin{cases} (x_{I_k})_i & \text{ if } i \in I_k, \\ 0 & \text{ otherwise.} \end{cases}$$
Observe that $\phi(v_{I_k}) = \phi(x_{I_k})$ for every $\phi \in E$ if $\supp(E) \subseteq I_k$.  In addition, if $\supp(E)$ does not intersect $I_k$, then $\supp(E) \cap I_k = \varnothing$ and $\phi(v_{I_k}) = 0$ for every $\phi \in E$.  In either case, $v_{I_k} \in L_E$ for every minimal locked set $E$ for $f^q$, and therefore $v_{I_k} \in V(f)$.  The set of all nonzero $v_{I_k}$ is a basis for $V(f)$.  

Let 
$$P(x) = \sum_{k = 1}^m \frac{1}{|I_k|} v_{I_k} v_{I_k}^T.$$
Then $P$ is an $\infty$-norm nonexpansive projection from $\R^n$ onto $V(f)$. This completes the proof for $p = \infty$. 
\end{proof}

The maximum order of a permutation on $n$-elements is Landau's function $g(n)$.  It is well known that $g(n) \le e^{n/e}$ which is less than $2^{n-1}$ for all $n \ge 3$.  Therefore Theorem \ref{thm:periodic} confirms that the periods of any periodic orbit of a real analytic $\infty$-norm nonexpansive map on $\R^n$ is strictly less than $2^n$ for all $n \ge 3$.  

\section{Open Questions}

Theorem \ref{thm:main} guarantees that the fixed point set of a real analytic polyhedral norm nonexpansive map is a Lipschitz manifold. In every example that we are aware of, the fixed point set is also a differentiable manifold. So it seems natural to ask the following. 

\begin{question} \label{question:one}
Suppose that $X$ is a polyhedral normed space, $f: X \rightarrow X$ is nonexpansive and real analytic, and $\Fix(f)$ is nonempty. Is $\Fix(f)$ always a differentiable manifold?
\end{question}

The following stronger assertion would resolve the previous question and lead to simpler proofs of both Theorems \ref{thm:main} and Theorem \ref{thm:periodic}.

\begin{question}
Suppose that $X$ is a polyhedral normed space, $f: X \rightarrow X$ is nonexpansive and real analytic, and $\Fix(f)$ is nonempty. Is  there a differentiable nonexpansive retract from $X$ onto $\Fix(f)$? 
\end{question}

Finally we ask whether Theorem \ref{thm:main} is true without assuming that the norm is either polyhedral or strictly convex?  

\bibliography{DW2}
\bibliographystyle{plain}

\end{document}